\documentclass[reqno]{amsart} 
\usepackage{amsfonts, amsmath, amsthm, amssymb, latexsym, graphicx}

\theoremstyle{plain}
\newtheorem{theorem}{Theorem}[section]
\newtheorem{corollary}[theorem]{Corollary}
\newtheorem{lemma}[theorem]{Lemma}
\newtheorem{proposition}[theorem]{Proposition}
\theoremstyle{definition}

\newtheorem{remark}[theorem]{Remark}
\theoremstyle{remark}

\numberwithin{equation}{section}

\title[Bessel processes in the freezing regime]{Central limit theorems
 for multivariate Bessel processes in the freezing regime}

\author{Michael Voit}
\address{Fakult\"at Mathematik, Technische Universit\"at Dortmund,
          Vogelpothsweg 87,
          D-44221 Dortmund, Germany}
\email{michael.voit@math.tu-dortmund.de}


\begin{document}
\date{\today}

\begin{abstract} Multivariate Bessel processes  $(X_{t,k})_{t\ge0}$ 
are classified via associated root systems and multiplicity constants $k\ge0$.
They describe the  dynamics of
 interacting particle systems of Calogero-Moser-Suther\-land type.
Recently,  Andraus, Katori, and
 Miyashita derived some weak laws of large numbers for $X_{t,k}$ for fixed times $t>0$ and $k\to\infty$.

In this paper we derive associated central limit theorems for the root systems
 of types A, B and D in an elementary way.
In most cases, the limits will be normal distributions, but in the B-case there are
freezing limits where  distributions associated with the root system A 
or  one-sided normal distributions  on half-spaces appear. Our results are connected
 to central limit theorems of Dumitriu and Edelman for $\beta$-Hermite and $\beta$-Laguerre ensembles.
\end{abstract}

\maketitle

\section{Introduction} 
The dynamics of integrable interacting particle systems of 
Calogero-Moser-Suther\-land type on the real line $\mathbb R$ with $N$ particles 
 can be described by certain time-homogeneous diffusion processes on 
suitable closed subsets of $\mathbb R^N$. These processes are often called (multivariate or interacting) 
Bessel- or Dunkl-Bessel processes; for the general background we refer to
 \cite{CGY}, \cite{GY}, \cite{R1}, \cite{R2}, 
 \cite{RV1}, \cite{RV2} as well as to \cite{An}, \cite{DF}, \cite{DV}.
These processes are classified via root systems and a finite number of  multiplicity parameters
which govern the interactions. We here  consider the  root systems
 of types $A_{N-1}$, $B_N$, and $D_N$.

Let us consider some details of the case $A_{N-1}$ first. Here
 we have a  multiplicity $k\in[0,\infty[$, the processes $(X_{t,k})_{t\ge0}$ 
live on the closed Weyl chamber
$$C_N^A:=\{x\in \mathbb R^N: \quad x_1\ge x_2\ge\ldots\ge x_N\},$$
 the generator of the transition semigroup is 
\begin{equation}\label{def-L-A} Lf:= \frac{1}{2} \Delta f +
 k \sum_{i=1}^N\Bigl( \sum_{j\ne i} \frac{1}{x_i-x_j}\Bigr) \frac{\partial}{\partial x_i}f ,
 \end{equation}
and we assume reflecting boundaries, i.e., the domain of the operator $L$ is
$$D(L):=\{f|_{C_N}: \>\> f\in C^{(2)}(\mathbb R^N), 
\>\>\> f\>\>\text{ invariant under all permutations of coordinates}\}.$$

We are interested in limit theorems for  $(X_{t,k})_{t\ge0}$
for fixed $t>0$ in freezing regimes, i.e., for $k\to\infty$.
For this we recall that by \cite{R1}, \cite{R2}, \cite{RV1}, \cite{RV2}, the 
 transition probabilities are given for $t>0$,  $x\in C_N$, $A\subset C_N$ a Borel set, by
\begin{equation}\label{density-A}
K_t(x,A)=c_k^A \int_A \frac{1}{t^{\gamma_A+N/2}} e^{-(\|x\|^2+\|y\|^2)/(2t)} J_k^A(\frac{x}{\sqrt{t}}, \frac{y}{\sqrt{t}}) 
\cdot w_k(y)\> dy
\end{equation}
with
\begin{equation}\label{def-wk-A-gamma}
w_k^A(x):= \prod_{i<j}(x_i-x_j)^{2k}, \quad\quad \gamma_A=kN(N-1)/2, \end{equation}
and the Macdonald-Mehta-Opdam   constant
\begin{equation}\label{const-A}
 c_k^A:= \Bigl(\int_{C_N^A}  e^{-\|y\|^2/2}\cdot \prod_{i<j} (y_i-y_j)^{2k} \> dy\Bigr)^{-1}
=\frac{N!}{(2\pi)^{N/2}} \cdot\prod_{j=1}^{N}\frac{\Gamma(1+k)}{\Gamma(1+jk)};
\end{equation}
see \cite{O} or \cite{Me}.
 Notice that
$w_k^A$ is homogeneous of degree $2\gamma_A$.
Moreover,
$J_k^A$ is a multivariate Bessel function of type $A$ with multiplicity $k$;
 see e.g. \cite{R1}, \cite{R2} and references there.
For the moment, we do not need much informations about $J_k^A$. We only recapitulate that 
$J_k^A$ is analytic on $\mathbb C^N \times \mathbb C^N $ with
$ J_k^A(x,y)>0$ for $x,y\in \mathbb R^N $, and with
 $J_k^A(x,y)=J_k^A(y,x)$ and $J_k^A(0,y)=1$
for  $x,y\in \mathbb C^N $.

If we start in $0\in \mathbb R^N$, then $X_{t,k}$ has the  density
\begin{equation}\label{density-A-0}
 \frac{c_k}{t^{\gamma+N/2}} e^{-\|y\|^2/(2t)} \cdot w_k(y)\> dy
\end{equation}
on $C_N^A$ for $t>0$, i.e.,
 $X_{t,k}/{\sqrt{tk}}$ has a  density of the form
$$const.(k)\cdot exp\Bigl( k\Bigl(2\sum_{i,j: i<j} \ln(y_i-y_j) -\|y\|^2/2\Bigr)\Bigr)=:
 const.(k)\cdot exp\Bigl( k\cdot W_A(y)\Bigr)$$
which is in particular well-known for $k=1/2, 1,2$ as the distribution of the ordered eigenvalues of
Gaussian orthogonal, unitary, and symplectic ensembles; see e.g. \cite{D}.
For general $k>0$ it is known from the tridiagonal $\beta$-Hermite ensembles of \cite{DE1}.
It is well-known (see \cite{AKM1} and also Section 6.7 of \cite{S}) that  $W_A$ is maximal on $C_N^A$  precisely
 for $y=\sqrt 2\cdot {\bf z}$ where $ {\bf z}\in C_N^A$ is the vector
with the zeros of the classical  Hermite polynomial $H_N$ as entries
where  the $(H_N)_{N\ge 0}$ are orthogonal w.r.t.
 the density  $e^{-x^2}$.
A saddle point argument thus immediately yields that
\begin{equation}\label{LLN-A-start-0}
\lim_{k\to\infty}\frac{X_{t,k}}{\sqrt{2tk}}= {\bf z} \quad\quad(k\to\infty)
\end{equation}
in distribution and thus in  probability whenever the $X_{t,k}$ are defined on a common probability space. 
It was shown in \cite{AKM1} that this even holds when we start in any fixed point $x\in C_N^A$ 
or even with some more or less arbitrary starting distribution.

We  prove a corresponding central limit theorem in an elementary way in Section 2.
This CLT was derived in \cite{DE2} by other methods via an interpretation through tridiagonal matrix models.
 We prove that for starting 
 in  $0\in C_N^A$ and any $t>0$, $X_{t,k} -\sqrt{2kt}\cdot {\bf z}$ converges
in distribution to some centered $N$-dimensional normal distribution with some covariance matrix which again contains
the zeros of $H_N$ as major ingredients; see Theorem \ref{clt-main-a} below. 
The proof  is based on the explicit  density of  $X_{t,k}$ above and
 elementary calculations which involve  the  zeros of $H_N$.
As a byproduct of the CLT we automatically get some determinantal formula for the  zeros of $H_N$ which
 is possibly new.

 We also  derive corresponding CLTs for the 
Bessel processes associated with the root systems $B_N$ and $D_N$. 
In the B-case, the multiplicity $k=(k_1,k_2)$
is 2-dimensional. Motivated by the LLNs in \cite{AKM2} and  \cite{AM} for the B-cases, we 
study central limit theorems for several freezing cases. We in particular study the case
 $(k_1,k_2)=(c\cdot \beta,\beta)$
with $c>0$ fixed and $\beta\to\infty$ in Section 3, but we also shall study the case  where  $k_2>0$ is fixed
and  $k_1\to \infty$ in  Section 4 as well as  the case  $k_1>0$ fixed and $k_2\to \infty$ in Section 6. 
It will turn out that in the first case we obtain again a classical normal distribution
 in the limit where the covariance matrix is formed in terms of the zeros of classical Laguerre polynomials,
 where the index of the polynomials depends on $k_2$. In the second regime,
 the limit distribution has a  density of type A as  in
(\ref{density-A-0}). 
 In the third case we shall get some one-sided normal distributions which live on a certain halfspace.
 Furthermore, Section 5 will be devoted to the root system  $D_N$. 
We point out that some of the limit results are available 
for arbitrary, fixed starting points and not just for the case with start in 0.

The Bessel processes are diffusions on Weyl chambers which satisfy some stochastic 
differential equations; see \cite{GY}, \cite{CGY}. These SDEs are used in \cite{AV}
 to derive  locally uniform  strong laws of large numbers  for $X_{t,k}$ for  $k\to\infty$ with  strong
rates of convergence, whenever the processes start in points of the form $\sqrt k\cdot x$ where $x$
 is some point in the interior of the Weyl chamber. It was possible to derive  a CLT
 in the  B-case for a particular freezing regime for these starting points in \cite{AV}.
Further CLTs for  starting points of the form $\sqrt k\cdot x$ are given in \cite{VW}.

\section{A central limit theorem for the root system $A_{N-1}$}

In this section we derive a CLT for Bessel processes of type A for $k\to\infty$ with
 a $N$-dimensional normal distribution as limit. The centerings as well as the entries of the covariance 
matrices of the limit will be described in terms of the zeros of the classical Hermite polynomials $H_N$. 
This connection is based on the following fact on these zeros, which is originally due to Stieltjes:

\begin{lemma}\label{char-zero-A}
 For $y\in C_N^A$, the following statements  are equivalent:
\begin{enumerate}
\item[\rm{(1)}] The function $W_A(x):=2\sum_{i,j: i<j} \ln(x_i-x_j) -\|x\|^2/2$ is maximal at $y\in C_N^A$;
\item[\rm{(2)}] For $i=1,\ldots,N$:  $\frac{1}{2}y_i= \sum_{j: j\ne i} \frac{1}{y_i-y_j}$;
\item[\rm{(3)}] The vector 
$${\bf z}:=(z_1,\ldots,z_N):=(y_1/\sqrt2, \ldots,y_N/\sqrt2)$$
 consists of
 the ordered zeros of the classical Hermite polynomial $H_N$.
\end{enumerate}
Furthermore, the vector ${\bf z}$ of (3) satisfies for $t>0$,
\begin{equation}\label{potential-z-A}
 -\frac{\|{\bf z}\|^2}{2t}+ 2\sum_{i<j} \ln(z_i-z_j)= -\frac{N(N-1)}{2}(1-\ln t)+ \sum_{j=1}^N j\ln j.
\end{equation}
\end{lemma}

\begin{proof} For the equivalence of (1)-(3)
see \cite{AKM1} or  Section 6.7 of \cite{S}. For (\ref{potential-z-A}) we
 refer to appendix D and the comments between Eqs.~(58) and (59) in \cite{AKM1}.
\end{proof}

Using the zeros of $H_N$ we now turn to the main result of this section:

\begin{theorem}\label{clt-main-a}
Consider the  Bessel processes $(X_{t,k})_{t\ge0}$ of type $A_{N-1}$ on $C_N^A$ for $k\ge0$ with start in $0\in C_N^A$ .
Then
$$\frac{X_{t,k}}{\sqrt t} -  \sqrt{2k}\cdot {\bf z}$$
converges for $k\to\infty$ to the centered $N$-dimensional distribution $N(0,t\cdot \Sigma)$
with the regular covariance matrix $\Sigma$ with $\Sigma^{-1}=(s_{i,j})_{i,j=1,\ldots,N}$ with
\begin{equation}\label{covariance-matrix-A}
s_{i,j}:=\left\{ \begin{array}{r@{\quad\quad}l}  1+\sum_{l\ne i} (z_i-z_l)^{-2} & \text{for}\quad i=j \\
   -(z_i-z_j)^{-2} & \text{for}\quad i\ne j  \end{array}  \right.  . 
\end{equation}
\end{theorem}

\begin{proof}
We first observe that by the definition of the transition kernels $K_t$ of our Bessel processes in (\ref{density-A}),
 the $K_t$ admit the same
space-time-scaling as Brownian motions, i.e., for all $t>0$, $x\in C_N^A$, and all Borel sets $A\subset C_N^A$,
$K_t(x,A)=K_1(\sqrt t\cdot x, \sqrt t\cdot A)$. We thus may assume that $t=1$ in the proof.

$X_{1,k}$ has the  density
$$c_k^A e^{-\|y\|^2/2}\cdot exp\Bigl( 2k \sum_{i<j}\ln(y_i-y_j)\Bigr)$$
on $C_N^A$. Hence, $X_{1,k}-\sqrt{2k}\cdot{\bf z}$ has the Lebesgue density
\begin{align}\label{a-density-detail}
f_k^A(y):=c_k^A & \cdot exp\Bigl(-\|y +\sqrt{2k}\cdot{\bf z}\|^2/2+2k \sum_{i<j}\ln\bigl(y_i-y_j+\sqrt{2k}(z_i-z_j)\bigl) \Bigr)
 \\
=c_k^A &\cdot exp\Bigl(-\|y\|^2/2-\sqrt{2k} \langle y,{\bf z}\rangle- k\|{\bf z}\|^2
+2k \sum_{i<j}\ln(\sqrt{2k}(z_i-z_j)) \Bigr)\times\notag\\
&\times  exp\Bigl(
2k \sum_{i<j}\ln\bigl(1+ \frac{y_i-y_j}{\sqrt{2k}(z_i-z_j)}\bigr)\Bigr) \notag
\end{align}
on the shifted cone $C_N^A-\sqrt{2k}\cdot{\bf z}$ with $f_k^A(y)=0$ otherwise on $\mathbb R^N$.
We now split this formula  into two parts
$$f_k^A(y)=\tilde c_k\cdot h_k(y),$$
where $h_k$ depends on $y$ and
 the remainder $\tilde c_k$ is   constant w.r.t.~$y$.
This constant term is 
\begin{align}
\tilde c_k:=& c_k^A  e^{-k\|{\bf z} \|^2}\cdot exp\Bigl(2k \sum_{i<j}\ln(\sqrt{2k}(z_i-z_j))\Bigr)\notag\\
=& c_k^A  exp\Bigl(-k\bigl( \|{\bf z} \|^2-2 \sum_{i<j}\ln(z_i-z_j)\bigr)\Bigr)\cdot (2k)^{kN(N-1)/2}\notag\\
=& c_k^A exp\Bigl(-k \frac{N(N-1)}{2}(1+\ln 2)+k\sum_{j=1}^N j\ln j\Bigr)\cdot (2k)^{kN(N-1)/2}
\notag\\
=& c_k^A  (k/e)^{kN(N-1)/2}\cdot\prod_{j=1}^N j^{kj}.\notag
\end{align}
Notice that the third $=$ above follows from (\ref{potential-z-A}) for $t=1/2$. Hence, by (\ref{const-A}),
\begin{equation}\label{a-density-constants}
\tilde c_k(x)=\frac{N!}{(2\pi)^{N/2}} \cdot\prod_{j=1}^{N}\frac{\Gamma(1+k)}{\Gamma(1+jk)}\cdot
   (k/e)^{kN(N-1)/2}\cdot\prod_{j=1}^N j^{kj}.
\end{equation}
 Stirling's formula $\Gamma(k+1)\sim \sqrt{2\pi k}(k/e)^k$ and elementary calculations now lead to 
\begin{equation}\label{a-density-constants-limit}
\lim_{k\to\infty}\tilde c_k=\frac{\sqrt{N!}}{(2\pi)^{N/2}}.
\end{equation}

We next turn to the factor $h_k(y)$; it is given by
$$h_k(y):=  exp\Bigl(-\|y\|^2/2 -\sqrt{2k} \langle y,{\bf z}\rangle +
2k \sum_{i<j}\ln\bigl(1+ \frac{y_i-y_j}{\sqrt{2k}(z_i-z_j)}\bigr)\Bigr)$$
By the power series  of $\ln(1+x)$,
\begin{equation}\label{ln-expansion}
\ln\bigl(1+ \frac{y_i-y_j}{\sqrt{2k}(z_i-z_j)}\bigr) = \frac{y_i-y_j}{\sqrt{2k}(z_i-z_j)}
 -\frac{(y_i-y_j)^2}{4k(z_i-z_j)^2} + O(k^{-3/2}).
\end{equation}
Furthermore, by part (2) of Lemma \ref{char-zero-A},
\begin{equation}\label{zero-equation-a}
-\sqrt{2k} \langle y,{\bf z}\rangle+\sqrt{2k} \sum_{i<j} \frac{y_i-y_j}{z_i-z_j}=
\sqrt{2k}\sum_{i=1}^N y_i\bigl(-z_i+ \sum_{j: \> j\ne i} \frac{1}{z_i-z_j}\bigr)=0.
\end{equation}
Therefore, 
\begin{equation}\label{density-a-limit-1}
h_k(y)= exp\Bigl(-\|y\|_2^2/2-\frac{1}{2}\sum_{i<j}\frac{(y_i-y_j)^2}{(z_i-z_j)^2} + O(k^{-1/2})\Bigr).
\end{equation}
Now let $f\in C_b(\mathbb R^N)$ be a bounded continuous function. We conclude from
(\ref{a-density-detail}), (\ref{a-density-constants-limit}),(\ref{density-a-limit-1}) that
\begin{align}\label{density-a-limit-2}
\lim_{k\to\infty}\int_{\mathbb R^N} f(y)\cdot f_k^A(y)\> dy =& \lim_{k\to\infty} \tilde c_k \int_{\mathbb R^N}  f(y)\cdot h_k(y)\> dy
 \\
=& \frac{\sqrt{N!}}{(2\pi)^{N/2}} \cdot
\int_{\mathbb R^N} f(y) e^{-\|y\|^2/2} exp\Bigl(-\frac{1}{2}\sum_{i<j}\frac{(y_i-y_j)^2}{(z_i-z_j)^2}\Bigr) \> dy.
\notag
\end{align}
For this we have to check that we may apply  dominated convergence. For this we again consider
 the Taylor polynomial  of $\ln(1+x)$ and notice that by the Lagrange remainder,
\begin{equation}\label{ln-expansion-remainder}
\ln\bigl(1+ \frac{y_i-y_j}{\sqrt{2k}(z_i-z_j)}\bigr) = \frac{y_i-y_j}{\sqrt{2k}(z_i-z_j)}
 -\frac{(y_i-y_j)^2}{4k(z_i-z_j)^2}\cdot w
\end{equation}
with some $w\in[0,1]$.  This implies readily  that 
we could apply  dominated convergence  in (\ref{density-a-limit-2}).

On the other hand, Eq.~(\ref{density-a-limit-2}) says that the probability measures with the densities $f_k^A$
tend weakly to the measure with Lebesgue density
\begin{equation}
\frac{\sqrt{N!}}{(2\pi)^{N/2}}\cdot
 e^{-\|y\|^2/2}\cdot exp\Bigl(-\frac{1}{2}\sum_{i<j}\frac{(y_i-y_j)^2}{(z_i-z_j)^2}\Bigr),
\end{equation}
which is automatically a probability measure. This measure is  necessarily the normal distribution claimed
 in the  theorem up to scaling such that we necessarily have the correct  normalization. This completes the proof.
\end{proof}

Notice that the final arguments in the proof  about the correct normalizations above  automatically lead to 
the following remarkable result for the zeros of the Hermite polynomial $H_N$:

\begin{corollary}\label{corr-determinant}
For each $N\in\mathbb N$ consider the ordered zeros $z_1\ge \ldots\ge z_N$ of the $N$-th Hermite polynomial $H_N$.
Form the matrix $S:=(s_{i,j})_{i,j=1,\ldots,N}$ with
\begin{equation}\label{covariance-matrix-A-2}
s_{i,j}:=\left\{ \begin{array}{r@{\quad\quad}l}  1+\sum_{l\ne i} (z_i-z_l)^{-2} & \text{for}\quad i=j \\
   -(z_i-z_j)^{-2} & \text{for}\quad i\ne j  \end{array}  \right.  . 
\end{equation}
Then $det\> S=N!$.
\end{corollary}

\begin{remark}
As mentioned in the introduction, Theorem \ref{clt-main-a} was derived in \cite{DE2}
 via the tridiagonal matrix models for $\beta$-Hermite ensembles
of Dumitriu and Edelman in \cite{DE1}. In fact, in Theorem  3.1
of \cite{DE2}, 
  Dumitriu and Edelman  obtain the CLT  above
with a direct, but quite complicated formula for the covariance matrix $\Sigma$ of the limit
 in terms of the zeros of $H_N$ combined with
 the Hermite polynomials $H_l$ ($l=1,\ldots,N$). It is quite unclear how the expression for 
 $\Sigma$ in \cite{DE2} 
corresponds to our formula for $\Sigma^{-1}$ above, i.e., the equality of these matrices may be seen as a further
corollary from Theorem \ref{clt-main-a}.
\end{remark}

\begin{remark}
One might try to extend the preceding proof to the case where the
 processes $(X_{t,k})_{t\ge0}$ start in some  starting point
 $x\in C_N^A$ with $x\ne0$. If $x$ has the form $x=c(1,\ldots,1)$, then
 $\frac{X_{t,k}}{\sqrt t} -  \sqrt{2k}\cdot {\bf z}-x$ again tends to
 $N(0,t\cdot \Sigma)$ in distribution with $\Sigma$ as above in the theorem.

This follows easily from   Theorem \ref{clt-main-a} and the fact that the kernels 
 $K_t$ are partially translation invariant in the sense that
\begin{equation}\label{translation-invariance-A}
K_t(x+c(1,\ldots,1), A+c(1,\ldots,1))= K_t(x,A) \quad\quad \text{for}\>\> c,t>0, \> x\in C_N^A,\> A\subset C_N^A.
\end{equation}
In fact this is a consequence from the well-known fact that the Bessel functions of type A satisfy
\begin{equation}\label{bessel-reduction-A}
J_k^A(x,y)=e^{N \bar x \cdot \bar y}\cdot\tilde J_k^A(\pi_N(x),\pi_N(y)) \quad\quad(x,y\in \mathbb R^N)
\end{equation}
with $\bar x:=\frac{1}{N}(x_1+\ldots+x_N)$ and with
 the orthogonal projection $\pi_N$ from $\mathbb R^N$ onto the orthogonal complement 
  ${\bf 1}^\bot:=(1,\ldots,1)^\bot$ of $\mathbb R\cdot(1,\ldots,1)$
w.r.t.~the standard scalar product on $\mathbb R^N$ and with the Bessel function $\tilde J_k$
 which is associated with the reduced irreducible root system of type $A_{N-1}$  on
 ${\bf 1}^\bot\equiv \mathbb R^{N-1}$. For the details we refer to \cite{BF}.

Moreover, for $N=2$, the reduced Bessel function $\tilde J_k$ is a classical (modified) one-dimensional 
Bessel function. 
This implies that here for all $x,y\in \mathbb R^{N-1}$, the limit
\begin{equation}\label{reduction-bessel-limit-A}
\lim_{k\to\infty}\tilde J_k^A(x,\sqrt k \cdot y)=:\phi(x,y)>0
\end{equation}
exists locally uniformly. 
With this limit, the proof of the CLT above can be extended to arbitrary starting points $x\in C_N^A$
 for $N=2$ where the limit does not depend on $x$. Such a result will also appear
 for some limit for the root systems $B_N$ with $N\ge1$ in Section 3.

It is unclear whether for
the A-case and $N\ge3$ a corresponding result to (\ref{reduction-bessel-limit-A}) holds.
 Possibly,
  (\ref{reduction-bessel-limit-A}) 
can be checked via induction on $N$ via the recursive integral representation of 
Amri \cite{A} for $ J_k^A$ which
 is a consequence of the integral representation of  Okounkov and Olshanski for Jack polynomials in \cite{OO}.
\end{remark}

\section{A central limit theorem for the root system $B_N$}

In this section we derive a first CLT for Bessel processes of type B. We recapitulate that
in the case $B_N$,  we have 2  multiplicities $k_1,k_2>0$, the processes live on 
$$C_N^B:=\{x\in \mathbb R^N: \quad x_1\ge x_2\ge\ldots\ge x_N\ge0\},$$
 the generator of the transition semigroup is 
\begin{equation}\label{def-L-B} Lf:= \frac{1}{2} \Delta f +
 k_2 \sum_{i=1}^N \sum_{j\ne i} \Bigl( \frac{1}{x_i-x_j}+\frac{1}{x_i+x_j}  \Bigr)
 \frac{\partial}{\partial x_i}f 
\quad + k_1\sum_{i=1}^N\frac{1}{x_i}\frac{\partial}{\partial x_i}f, \end{equation}
and we again assume reflecting boundaries. Similar to  the A-case in the introduction,  we have the transition 
probabilities
\begin{equation}\label{density-general-b}
K_{t,k}(x,A)=c_k^B \int_A \frac{1}{t^{\gamma_B+N/2}} e^{-(\|x\|^2+\|y\|^2)/(2t)} J_k^B(\frac{x}{\sqrt{t}}, \frac{y}{\sqrt{t}}) 
\cdot w_k^B(y)\> dy
\end{equation}
with
\begin{equation}\label{def-wk-b}
w_k^B(x):= \prod_{i<j}(x_i^2-x_j^2)^{2k_2}\cdot \prod_{i=1}^N x_i^{2k_1},\end{equation}
 $ \gamma_B=k_2N(N-1)+k_1N$, and with the Macdonald-Mehta-Opdam-type normalization
\begin{align}\label{const-b}
 c_k^B:=& \Bigl(\int_{C_N^B}  e^{-\|y\|^2/2} w_k^B(y) \> dy\Bigr)^{-1} \\
=&\frac{N!}{2^{N(k_1+(N-1)k_2-1/2)}} \cdot\prod_{j=1}^{N}\frac{\Gamma(1+k_2)}{\Gamma(1+jk_2)\Gamma(\frac{1}{2}+k_1+(j-1)k_2)};
\notag\end{align}
see \cite{O}. Again
$w_k^B$ is homogeneous of degree $2\gamma_B$, and
$J_k^B$ is a multivariate Bessel function of type  $B$ with multiplicities $k:=(k_1,k_2)$.
 
We now study CLTs fo several freezing regimes in this section as well as in Sections 4 and 6.
 We here start with the case 
$(k_1,k_2)=(\nu\cdot\beta,\beta)$ with $\nu>0$ fixed and $\beta\to\infty$. Laws of large numbers in this case
can be found in in \cite{AKM2}, \cite{AV} where there in the limit now
 the zeros of the classical Laguerre polynomials
 $L_N^{(\nu-1)}$ appear. We recapitulate that the  $L_N^{(\nu-1)}$ are orthogonal  w.r.t. the density $e^{-x}\cdot x^{\nu-1}$ on $]0,\infty[$
for $\nu>0$.
We need the following known facts about the zeros of $L_N^{(\nu-1)}$.

\begin{lemma}\label{char-zero-B1}
Let $\nu>0$. For $r\in C_N^B$, the following statements  are equivalent:
\begin{enumerate}
\item[\rm{(1)}] The function 
$$W_B(y):=2\sum_{ i<j} \ln(y_i^2-y_j^2) +2\nu \sum_{i}\ln y_i-\|y\|^2/2$$
 is maximal at $r\in C_N^B$;
\item[\rm{(2)}] For $i=1,\ldots,N$, $r=(r_1,\ldots,r_N)$ satisfies
$$\frac{1}{2}r_i= \sum_{j: j\ne i} \frac{2r_i}{r_i^2-r_j^2} +\frac{\nu}{r_i}=\sum_{j: j\ne i} \Bigl(
\frac{1}{r_i-r_j} +\frac{1}{r_i+r_j}\Bigr) +\frac{\nu}{r_i} ;$$ 
\item[\rm{(3)}] If $z_1^{(\nu-1)}\ge\ldots\ge z_N^{(\nu-1)}$ are the ordered zeros of $L_N^{(\nu-1)}$, then 
\begin{equation}\label{y-max-B1}
2(z_1^{(\nu-1)},\ldots, z_N^{(\nu-1)})= (r_1^2, \ldots, r_N^2).
\end{equation}
\end{enumerate}
The vector $r$ of (1)-(3) satisfies
\begin{align}\label{equality-F-B}
-\frac{1}{2}\|r\|^2 +&\nu\sum_{j=1}^N \ln r_j^2 + 2\sum_{i<j}\ln ( r_i^2 -r_j^2)= \\
&=N(N+\nu-1)(-1+\ln 2)+ \sum_{j=1}^N j\ln j + \sum_{j=1}^N(\nu +j-1) \ln(\nu +j-1)
\notag
\end{align}
\end{lemma}

\begin{proof}
For the equivalence of (1)-(3) we refer to \cite{AKM2}; see in particular Appendix C there.
 Moreover, this equivalence is more or less also contained in Section 6.7 of \cite{S}.
For the proof of (\ref{equality-F-B}) we also refer to  \cite{AKM2}.
 In fact, one has to compare Eq.~(12)  with the comments on between (75) and (76) there. 
Please notice that the definitions of $(\beta,\nu)$ here 
are slightly different from that in \cite{AKM2}; 
in $\beta$ there is a multiplicative factor 2, and in $\nu$ there is a shift by $1/2$.
\end{proof}

In order to 
handle arbitrary starting points $x$, we need the following asymptotic result for the Bessel functions of type $B$;
see Lemma 5 of \cite{AKM2} and notice that our notions of $\nu,\beta$ are different from  \cite{AKM2} as described above:

\begin{lemma}\label{asymptotic-bessel-B1}
For all $x,y\in C_N^B$ and $\nu>0$,
$$\lim_{\beta\to\infty} J_{(\nu\cdot \beta, \beta)}^B(\sqrt \beta\cdot x, y)= 
 exp\Bigl( \frac{\|x\|_2^2\|y\|_2^2}{4N(\nu+N-1)}\Bigr).$$
This limit holds locally uniformly in $x,y$.
\end{lemma}

We now turn to the  main result of this section:

\begin{theorem}\label{clt-main-b1}
Fix some starting point $x$ in  the Weyl chamber $C_N^B$, and 
consider the associated Bessel processes $(X_{t,k})_{t\ge0}$ of type $B_N$ on $C_N^B$ for $k=(k_1,k_2)$.
Then, for the vector $r\in C_N^B$ of Lemma \ref{char-zero-B1},
$$\frac{X_{t,(\nu\cdot\beta,\beta)}}{\sqrt t} -  \sqrt{\beta }\cdot r$$
converges for $\beta\to\infty$ to the centered $N$-dimensional distribution $N(0,t\cdot \Sigma)$
with the regular covariance matrix $\Sigma$ with  $\Sigma^{-1}=(s_{i,j})_{i,j=1,\ldots,N}$ with
\begin{equation}\label{covariance-matrix-B}
s_{i,j}:=\left\{ \begin{array}{r@{\quad\quad}l}  1+ \frac{2\nu}{r_i^2}+2\sum_{l\ne i} (r_i-r_l)^{-2}+2\sum_{l\ne i} (r_i+r_l)^{-2} &
 \text{for}\quad i=j \\
 2(r_i+r_j)^{-2}  -2(r_i-r_j)^{-2} & \text{for}\quad i\ne j  \end{array}  \right.  . 
\end{equation}
\end{theorem}

\begin{proof} As in the A-case we may assume that $t=1$ without loss of generality. Let $k=(\nu\cdot\beta,\beta)$.
Taking the starting point $x$ into account,
$X_{1,k}$ has the  density
$$c_k^B e^{-\|x\|^2/2-\|y\|^2/2}\cdot J_k^B(x,y)
\cdot exp\Bigl( 2\beta \sum_{i<j}\ln(y_i^2-y_j^2) + 2\nu\beta \sum_{i=1}^N \ln y_i\Bigr)$$
on $C_N^B$. Hence, $X_{1,k}-\sqrt{\beta}\cdot r$ has the density
\begin{align}\label{b-density-detail}
f_\beta^B(y):=c_k^B & e^{-\|x\|^2/2} J_k^B(x,y+\sqrt{\beta}\cdot r)e^{-\|y+\sqrt{\beta}\cdot r \|^2/2}\times \\
 \times exp\Bigl(&2\beta \sum_{i<j}\ln\bigl((y_i+\sqrt{\beta}\cdot r_i)^2- (y_j+\sqrt{\beta}\cdot r_j)^2\bigr)
+ 2\nu\beta \sum_{i=1}^N \ln (y_i+\sqrt{\beta}\cdot r_i)\Bigr)\notag\\
=c_k^B & e^{-\|x\|^2/2} J_k^B(x,y+\sqrt{\beta}\cdot r)e^{-\|y\|^2/2} e^{-\beta\|r\|^2/2}  e^{-\sqrt{\beta} \langle y,r\rangle}
\times \notag\\
 &\times  exp\Bigl(2\beta \sum_{i<j}\ln\bigl(1+ \frac{y_i-y_j}{\sqrt{\beta}(r_i-r_j)}\bigr)+
2\beta \sum_{i<j}\ln\bigl(1+ \frac{y_i+y_j}{\sqrt{\beta}(r_i+r_j)}\bigr)\Bigr) \times \notag\\
 &\times  exp\Bigl(2\nu\beta \sum_{i=1}^N \ln(1+ \frac{y_i}{\sqrt{\beta} r_i})\Bigr) 
 exp\Bigl(2\nu\beta\sum_{i=1}^N \ln(\sqrt{\beta} r_i)\Bigr)\times \notag\\
 &\times
 exp\Bigl(2\beta \sum_{i<j}\ln(\sqrt{\beta}( r_i-r_j)) +2\beta \sum_{i<j}\ln(\sqrt{\beta}( r_i+r_j))\Bigr)
 \notag
\end{align}
on the shifted cone $C_N^B-\sqrt{\beta}\cdot r$ with $f_\beta^B(y)=0$ otherwise on $\mathbb R^N$.
We now split this formula  into two parts
$$f_\beta^B(y)= \tilde c_\beta \cdot h_\beta(y),$$
 where $ h_\beta(y)$ depends on $y$, and where
 the remainder $\tilde c_\beta$   is  constant w.r.t.~$y$.
The part depending on $y$ is given by
\begin{align}&h_\beta(y):=  exp\Bigl(-\|y\|^2/2 -\sqrt{\beta} \langle y,r\rangle +
2\beta \sum_{i<j}\ln\bigl(1+ \frac{y_i-y_j}{\sqrt{\beta}(r_i-r_j)}\bigr)\Bigr)\times\notag\\
 &\times
 exp\Bigl(2\beta \sum_{i<j}\ln\bigl(1+ \frac{y_i+y_j}{\sqrt{\beta}(r_i+r_j)}\bigr)+
2\nu\beta\sum_{i=1}^N \ln(1+ \frac{y_i}{\sqrt{\beta} r_i})\Bigr)\cdot  J_k^B(x,y+\sqrt{\beta}\cdot r).
\notag\end{align}
By the power series  of $\ln(1+x)$,
\begin{equation}\label{ln-expansion-b1}
\ln\bigl(1+ \frac{y_i\pm y_j}{\sqrt{\beta}(r_i\pm r_j)}\bigr) = \frac{y_i\pm y_j}{\sqrt{\beta}(r_i\pm r_j)}
 -\frac{(r_i\pm r_j)^2}{2\beta(r_i\pm r_j)^2} + O(\beta^{-3/2})
\end{equation}
and
\begin{equation}\label{ln-expansion-b2}
\ln(1+ \frac{y_i}{\sqrt{\beta} r_i})= \frac{y_i}{\sqrt{\beta}r_i}-\frac{ y_i^2}{2\beta r_i^2}+ O(\beta^{-3/2}).
\end{equation}
Furthermore, by part (2) of Lemma \ref{char-zero-B1},
\begin{align}\label{zero-equation-b}
-\sqrt{\beta}& \langle y,r\rangle+2\sqrt{\beta} \sum_{i<j} \frac{y_i-y_j}{r_i-r_j} +2\sqrt{\beta} \sum_{i<j} \frac{y_i+y_j}{r_i+r_j} 
+ 2\nu\sqrt{\beta}\sum_{i=1}^N \frac{y_i}{r_i}\\
&= \sqrt{\beta} \sum_{i=1}^Ny_i\Bigl( -r_i+ \sum_{j: \> j\ne i} \frac{1}{r_i-r_j}+ \sum_{j: \> j\ne i}\frac{1}{r_i+r_j}+ \frac{2\nu}{r_i}
\Bigr)=0.
\notag
\end{align}
Therefore, by (\ref{ln-expansion-b1})-(\ref{zero-equation-b}),
\begin{align}\label{density-b-limit-1}
h_\beta(y)=& exp\Bigl(-\|y\|_2^2/2-\sum_{i<j}\frac{(y_i-y_j)^2}{(r_i-r_j)^2} 
-\sum_{i<j}\frac{(y_i+y_j)^2}{(r_i+r_j)^2}\Bigr) \times\\
&\times exp\Bigl(-\nu  \sum_{i=1}^Ny_i^2/r_i^2 + O(\beta^{-1/2})\Bigr)\cdot 
 J_k^B(x,y+\sqrt{\beta}\cdot r).
\notag
\end{align}
We next observe that by Lemma \ref{asymptotic-bessel-B1},
\begin{align}\label{limit-besselfunction-b2}
\lim_{\beta\to\infty} J_{(\nu\cdot \beta, \beta)}^B( x,y+\sqrt{\beta}\cdot r )&=
\lim_{\beta\to\infty} J_{(\nu\cdot \beta, \beta)}^B( x,\sqrt{\beta}(r+ y/\sqrt{\beta}) )\notag\\
&=
 exp\Bigl( \frac{\|x\|_2^2\|r\|_2^2}{4N(\nu+N-1)}\Bigr)=: d_\nu(x).
\end{align}
In summary,
\begin{equation}\label{summary-b1}
\lim_{\beta\to\infty}h_\beta(y)=d_\nu(x) exp\Bigl(-\frac{\|y\|_2^2}{2}-\sum_{i<j}\frac{(y_i-y_j)^2}{(r_i-r_j)^2} 
-\sum_{i<j}\frac{(y_i+y_j)^2}{(r_i+r_j)^2}-\nu  \sum_{i=1}^N\frac{y_i^2}{r_i^2}\Bigr).
\end{equation}

We next study the second factor in the density $f_\beta^B(y)$ in
(\ref{b-density-detail}) which is independent of $y$. This constant is given by
\begin{align}\label{details-const-b11}
\tilde c_\beta:=& c_k^B e^{-\|x\|_2^2/2- \beta\|r\|_2^2/2}  exp\Bigl(2\nu\beta\sum_{i=1}^N\ln(\sqrt\beta \cdot r_i)\Bigr)\times
\notag\\
&\times
 exp\Bigl(2\beta \sum_{i<j}\ln(\sqrt{\beta}( r_i-r_j)) +2\beta \sum_{i<j}\ln(\sqrt{\beta}( r_i+r_j))\Bigr)
 \notag\\
=&  c_k^B exp\Bigl(\beta\Bigl(-\frac{\|r\|_2^2}{2}+2\nu\sum_{i=1}^N\ln r_i +2 \sum_{i<j}(\ln( r_i-r_j)+\ln( r_i+r_j))\Bigr)\Bigr)
 \times\notag\\
&\quad \times
 \beta^{\nu\beta N + \beta N(N-1)}\cdot e^{-\|x\|_2^2/2}
 \notag\notag\\
=&  c_k^B exp\Bigl(\beta\Bigl(
N(N+\nu-1)(-1+\ln 2)+ \sum_{j=1}^N j\ln j + \sum_{j=1}^N(\nu +j-1) \ln(\nu +j-1)\Bigr)\Bigr)
 \times\notag\\
&\quad \times  \beta^{\nu\beta N + \beta N(N-1)}\cdot e^{-\|x\|_2^2/2}.
 \end{align}
Notice that the last equation  above follows from (\ref{equality-F-B}).
We next study the constant  $c_k^B$. We conclude from 
 (\ref{const-b}), Stirling's formula $\Gamma(k+1)\sim \sqrt{2\pi k}(k/e)^k$,  from
$$\frac{\Gamma(k+1/2)}{\Gamma(k+1)}\sim \frac{1}{\sqrt k}$$
for $k\to\infty$, and from a longer elementary calculation that
\begin{align}
 c_k^B \sim &\frac{exp(\beta N(N+\nu-1))\cdot \sqrt{N!}}{ 2^{N\nu\beta +N(N-1)\beta -N/2} \cdot (2\pi)^{N/2}}\times\notag\\
&\quad\times \frac{1}{ \prod_{j=1}^{N} j^{j\beta}\cdot \prod_{j=1}^{N} (\nu-1+j)^{\beta(\nu-1+j)} 
\beta^{N(N-1)\beta -N/2 +\beta\nu N}}.\notag
\end{align}
If we plug this into (\ref{details-const-b11}), we see that
\begin{equation}\label{b-density-constants-limit}
\lim_{\beta\to\infty}\tilde c_\beta= \frac{e^{-\|x\|_2^2/2}2^{N/2}\cdot \sqrt{N!}}{ (2\pi)^{N/2}}.
\end{equation}

Now let $f\in C_b(\mathbb R^N)$ be a bounded continuous function. We conclude from
(\ref{b-density-detail}), (\ref{b-density-constants-limit}), and (\ref{summary-b1}) that
\begin{align}\label{density-b-limit-2}
\lim_{\beta\to\infty}&\int_{\mathbb R^N} f(y)\cdot f_\beta^B(y)\> dy = 
\lim_{\beta\to\infty}\Bigl( \tilde c_\beta \int_{\mathbb R^N}  f(y)\cdot h_\beta(y)\> dy\Bigr)
 \\
=& \frac{e^{-\|x\|_2^2/2} d_\nu(x)  2^{N/2}\cdot \sqrt{N!}}{  (2\pi)^{N/2}} \cdot\notag\\
&\cdot \int_{\mathbb R^N} f(y) \> exp\Bigl(-\Bigl(\frac{\|y\|^2}{2}+\sum_{i<j}\frac{(y_i-y_j)^2}{(r_i-r_j)^2}
+\sum_{i<j}\frac{(y_i+y_j)^2}{(r_i+r_j)^2}+ \nu \sum_{i=1}^N \frac{y_i^2}{r_i^2}
\Bigr)\Bigr) \> dy.
\notag
\end{align}
For this we have to check that we may apply  dominated convergence. For this we again consider
 the Taylor polynomial  of $\ln(1+x)$ and notice that by the Lagrange remainder,
\begin{equation}\label{ln-expansion-remainder-b1}
\ln\bigl(1+ \frac{y_i\pm y_j}{\sqrt{\beta}(z_i\pm z_j)}\bigr) = \frac{y_i\pm y_j}{\sqrt{\beta}(z_i\pm z_j)}
 -\frac{(y_i\pm y_j)^2}{2\beta(z_i\pm z_j)^2}\cdot w_\pm
\end{equation}
with some $w_\pm\in[0,1]$. Moreover, by the same reason,
\begin{equation}\label{ln-expansion-remainder-b2}
\ln(1+ \frac{y_i}{\sqrt{\beta} r_i})= \frac{y_i}{\sqrt{\beta}r_i}-\frac{ y_i^2}{2\beta r_i^2}\cdot w.
\end{equation}
with some $w\in[0,1]$. We here also have to estimate the Bessel functions $J_k^B$. For this
we recapitulate from \cite{RV2} that for all $k$ and all $x,y\in C_N^B$,
$$ J_k^B(x,y+\sqrt{\beta}\cdot r)\le exp(\langle x,y+\sqrt{\beta}\cdot r\rangle).$$
This shows that
\begin{equation}\label{est-j-b1}
J_k(x,y+\sqrt{\beta}\cdot r)\le e^{2\langle x,y\rangle} \quad\quad\text{for}\>\> \beta>0,\> y\in C_N^B
\>\> \text{with }\>\>  \langle x,y\rangle\ge \langle x,\sqrt{\beta}\cdot r \rangle.
\end{equation}
On the other hand, as $h_\beta(y+\sqrt{\beta}\cdot r)>0$ only occurs for $ y/\sqrt{\beta}+r\in C_N^B$, and as
the set 
$$\{w\in  C_N^B:\> \langle x,w\rangle\le \langle x,r \rangle\}$$
is compact, we obtain from Lemma \ref{asymptotic-bessel-B1}
that
$$\sup_{y\in\mathbb  R^N:\>  y/\sqrt{\beta}+ r\in C_N^B, \langle x, y/\sqrt{\beta}\rangle\le \langle x,r \rangle}
J_k(x,\sqrt{\beta}(y/\sqrt{\beta}+r))$$
is bounded. This estimation, (\ref{est-j-b1}), (\ref{ln-expansion-remainder-b1}), and  
 (\ref{ln-expansion-remainder-b2}) readily imply that 
 the dominated convergence theorem in (\ref{density-b-limit-2}) works as claimed.

On the other hand, Eq.~(\ref{density-b-limit-2}) says that the probability measures with the densities $f_\beta^A$
tend weakly to the measure with Lebesgue density
\begin{equation}
\frac{e^{-\|x\|_2^2/2} d_\nu(x)  2^{N/2}\cdot \sqrt{N!}}{  (2\pi)^{N/2}} \cdot
exp\Bigl(-\Bigl(\frac{\|y\|^2}{2}+\sum_{i<j}\frac{(y_i-y_j)^2}{(r_i-r_j)^2}
+\sum_{i<j}\frac{(y_i+y_j)^2}{(r_i+r_j)^2}+ \nu \sum_{i=1}^N \frac{y_i^2}{r_i^2}
\Bigr)\Bigr),
\end{equation}
which is  a probability measure by our results. This measure is  necessarily the normal distribution claimed
 in the  theorem with the correct  normalization. This completes the proof.
\end{proof}

Notice that the final arguments in the proof above for $x=0$
 about the correct normalizations above  automatically lead to 
the following remarkable result on the zeros of the Laguerre
polynomial $L_N^{(\nu-1)}$:

\begin{corollary}\label{corr-determinant-b}
For $N\in\mathbb N$ and $\nu>0$  consider the ordered zeros  $z_1^{(\nu-1)}\ge\ldots\ge z_N^{(\nu-1)}>0$ of the Laguerre
polynomial $L_N^{(\nu-1)}$. Let $r_i:= \sqrt{2 z_i^{(\nu-1)}}$ for $i=1,\ldots,N$, and
form the matrix $S:=(s_{i,j})_{i,j=1,\ldots,N}$ with
\begin{equation}\label{covariance-matrix-B-corr}
s_{i,j}:=\left\{ \begin{array}{r@{\quad\quad}l}  1+ \frac{2\nu}{r_i^2}+2\sum_{l\ne i} (r_i-r_l)^{-2}+2\sum_{l\ne i} (r_i+r_l)^{-2} &
 \text{for}\quad i=j \\
 2(r_i+r_j)^{-2}  -2(r_i-r_j)^{-2} & \text{for}\quad i\ne j  \end{array}  \right.  . 
\end{equation}
Then $det\> S= N!\cdot 2^N$.
\end{corollary}

\begin{remark}
 In the case that the processes start in the origin, 
Theorem \ref{clt-main-b1} was derived in \cite{DE2} via the tridiagonal matrix models
 for $\beta$-Laguerre ensembles
of Dumitriu and Edelman where in Theorem 4.1
of \cite{DE2} a direct, quite complicated formula for
 the covariance matrix $\Sigma$ of the limit in terms of the zeros of $L_N^{(\nu-1)}$ and
 the Laguerre polynomials $L_l^{(\nu-1)}$ ($l=1,\ldots,N$) is given. 
As in the Hermite case in Section 2 it is unclear how the expression for  $\Sigma$ in \cite{DE2} 
corresponds to our formula for $\Sigma^{-1}$ above.
\end{remark}

\begin{remark}
If we analyze the constants in the end of the proof of Theorem \ref{clt-main-b1} for arbitrary $x\in C_N^B$,
 we obtain that
$$e^{-\|x\|^2/2}\cdot exp\Bigl(\frac{ \|x\|^2\|r\|^2}{4N(N+\nu-1)}\Bigr)=1$$
with the vector $r$ from Lemma \ref{char-zero-B1}.
This means that $\|r\|^2=2N(N+\nu-1)$. In fact, if
 we translate this equation via (\ref{y-max-B1}) into a corresponding formula for the zeros of
 $L_N^{(\nu-1)}$, then we  just obtain Eq.~(C.10) in \cite{AKM2}.
\end{remark}

Theorem \ref{clt-main-b1}  can be easily extended from fixed starting points $x\in  C_N^B$ to arbitrary starting distributions
 $\mu\in M^1( C_N^B)$:

\begin{corollary}\label{corr-general-starting-b}
Let $\mu\in  M^1( C_N^B)$ be an arbitrary starting distribution and $\nu>0$. Consider the Bessel processes
$(X_{t,(\nu\beta,\beta)})_{t\ge0}$ of type B on  $C_N^B$ with starting distribution $\mu$.
Then, for each $t>0$ and with the vector $r\in  C_N^B$ and the
 normal distribution $N(0,t\cdot\Sigma)$ of
Theorem \ref{clt-main-b1},
$$ \frac{X_{t,(\nu\beta,\beta)}}{\sqrt t}-\sqrt{\beta }\cdot r \to N(0,t\cdot\Sigma)
 \quad \text{in distribution for}\quad \beta\to\infty.$$
\end{corollary}

\begin{proof}
Let $f\in C_b(\mathbb R^N)$ be a bounded continuous function, and let $t>0$ fixed. Using the kernels $K_{t,\beta}$ of (\ref{density-general-b}),
we obtain for the distributions $P_\beta\in  M^1(\mathbb R^N)$ of $ X_{t,(\nu\beta,\beta)}-\sqrt{\beta t}\cdot r$ that
\begin{align}
\int_{\mathbb R^N} f\> dP_\beta&= \int_{C_N^B}\Bigl( \int_{C_N^B} f(y-\sqrt{\beta t}\cdot r)\> K_{t,(\nu\beta,\beta)}(x,dy)\Bigr) d\mu(x)
\notag\\
&=: \int_{C_N^B} (T_\beta f)(x)\> d\mu(x)\notag
\end{align}
where $ (T_\beta f)(x)\to \int_{\mathbb R^N} f\> dN(0,t\cdot\Sigma) $ holds for $\beta\to\infty$  and
 all $x\in  C_N^B$  by Theorem  \ref{clt-main-b1}.
As $\| T_\beta f\|_\infty\le \|  f\|_\infty$, dominated convergence shows that
 $$\int_{\mathbb R^N} f\> dP_\beta\to \int_{\mathbb R^N} f\> dN(0,t\cdot\Sigma)$$
 for all  $f\in C_b(\mathbb R^N)$. This proves the claim.
\end{proof}

\section{A second central limit theorem for the root system $B_{N}$}

In this section we derive a second CLT for Bessel processes of type B with parameters $k=(k_1,k_2)$.
We here fix  $k_2$ and consider the case  $k_1\to\infty$.
We proceed similar to the preceding section for the first CLT in the case B 
and include the case that we have 
an arbitary fixed starting point $x\in C_N^B$.

In order to handle  arbitrary starting points $x$, we need the following asymptotic estimation
 for the Bessel functions of type $B$. For the proof we refer to Lemma 7
 in \cite{AKM2} where there a slightly different notation with $(\nu\beta,\beta)=(k_1,k_2)$ is used, 
and where in the right hand side limit in  Lemma 7 in \cite{AKM2} the terms $\sqrt 2$ 
have to be replaced by $2$:

\begin{lemma}\label{B-A convergence} Let $k_2>0$. Then, 
$$ \lim_{k_1\to\infty} J^B_{(k_1,k_2)}(\sqrt{k_1}x, y)= 
J^A_{k_2}(x^2/2,y^2/2)$$
 locally uniformly in $x,y\in  C_N^B$ where we use the notation $x^2=:(x_1^2,...,x^2_N)\in C_N^B$.
\end{lemma}

We notice that Lemma \ref{B-A convergence}  was derived for $x,y\in i\cdot \mathbb R^N$ 
 with precise estimates for the rate of convergence in \cite{RV3}.

We now turn to the main result of this section, a CLT 
where the limit distribution is a distribution of type A as studied in Section 2:

\begin{theorem}\label{Theorem-b-2}
For any fixed starting point $x\in C^B_N$ consider the Bessel processes $(X_{t,(k_1,k_2) })_{t\geq 0}$
 of type $B_N$ on  $C^B_N$. Consider the vector  $\mathbf{1}:=(1,...,1)\in \mathbb{R}^N$.
 Then, for all  $k_2>0$ and $t>0$, 
 $$ X_{t,(k_1,k_2)}-\sqrt{2t \cdot k_1}\cdot\mathbf{1}$$ 
converges  for $k_1\to \infty$ in distribution to
 $X^A_{t/2,k_2}$, where $(X^A_{s,k_2})_{s\ge 0}$ is a Bessel process of type A starting in the origin.
\end{theorem}

\begin{proof}
As in the proofs of the preceding CLTs we may assume
 that $t=1$ by the same self-similarity property of all involved processes.
We recapitulate that  $X_{1,(k_1,k_2) }$ has the density 
$$c_{(k_1,k_2)}^B e^{-\|x\|^2/2-\|y\|^2/2 }\cdot J^B_{(k_1,k_2) }(x,y)\cdot\prod_{i<j}(y^2_i-y^2_j)^{2k_2}
\cdot exp\left(2k_1\sum_{i=1}^N\ln y_i\right)$$
for $y$ in the interior of $ C^B_N$ where this density is equal to 0 otherwise on 
 $\mathbb{R}^N$, and where
we use the notations from Section 3. 
Hence, $X_{1,(k_1,k_2)}-\sqrt{ 2 k_1}\cdot\mathbf{1}$ has the density 
\begin{align}
f^B_{(k_1,k_2)}(y):&=c^B_{(k_1,k_2)}\cdot e^{-\|x \|^2/2}\cdot e^{\frac{-\|y+\sqrt{2k_1}\cdot\mathbf{1}\|^2}{2}}
\cdot J^B_{(k_1,k_2)}(x,y+\sqrt{ 2k_1}\cdot\mathbf{1})\cdot
\notag\\
& \cdot
\prod_{i<j}((y_i+\sqrt{2k_1})^2-(y_j+\sqrt{2k_1})^2)^{2k_2}
\cdot
 exp\left( 2k_1\sum _{i=1}^N \ln \left(y_i+\sqrt{2 k_1}\right)\right)
\notag\\
&=:\> c(k_1,k_2,x,N)\cdot\prod_{i<j}(y_i-y_j)^{2k_2} \tilde f_{k_1,k_2,x}(y)
\notag\end{align}
with
\begin{align}
\tilde f_{k_1,k_2,x}(y):=&
\prod_{i<j}\left(\frac{y_i+y_j+2\sqrt{2 k_1}}{2\sqrt{2 k_1}}\right)^{2k_2}
\cdot J^B_{(k_1,k_2)}(x,y+\sqrt{2 k_1}\cdot\mathbf{1})
\times
\notag\\
&\quad \times
exp\left(-\|y\|^2/2-\sqrt{2 k_1}\cdot \sum_{i=1}^{N}y_i 
+2k_1\sum _{i=1}^N 
\ln(1+\frac{y_i}{\sqrt{2 k_1}})\right)
\notag\end{align}
on the shifted cone $C^B_N-\sqrt{2 k_1}\cdot\mathbf{1}$, where $f^B_{(k_1,k_2)}(y)=0$
 otherwise on $\mathbb{R}^N$.
Here $c(k_1,k_2,x,N)>0$ is some constant which may be computed explicitly.

As in the proofs of the preceding CLTs we conclude  from the Taylor expansion of $\ln(1+z)$ that
$$\ln \left(1+\frac{y_i}{\sqrt{2k_1}}\right)=
 \frac{y_i}{\sqrt{2k_1}}-\frac{y_i^2}{4k_1}+O(k_1^{-3/2}).$$
This shows that
\begin{equation}\label{powerseries2}
exp\left(-\sqrt{2 k_1}\cdot \sum_{i=1}^{N}y_i 
+2k_1\sum _{i=1}^N 
\ln(1+\frac{y_i}{\sqrt{2 k_1}})\right)=exp\left(-\sum _{i=1}^N \frac{y_i^2}{2} +O(k_1^{-1/2})\right). 
\end{equation}
Moreover, we observe from  Lemma \ref{B-A convergence} and  Eq.~(\ref{bessel-reduction-A})  that
 \begin{align} \label{Bessel function approximation}
 \lim_{k_1\to\infty}J^B_{(k_1,k_2)}(x,y+\sqrt{2k_1}\cdot\mathbf{1})
=& \lim_{k_1\to\infty}J_{(k_1,k_2)}^B(x, \sqrt{2k_1}(\mathbf{1}+y/\sqrt{2k_1}))\notag\\
 =&J^A_{k_2}(x^2,\frac{1}{2}\cdot \mathbf{1})= e^{\|x\|^2/2 }.
 \end{align}
 In summary, we obtain that for all $x\in C_N^B$ and all $y$  in the limit set $C_N^A$ of the shifted cones
$ C_N^B-\sqrt{2k_1}\cdot \mathbf{1}$ for $k_1\to\infty$ that
$$\lim_{k_1\to\infty}\tilde f_{k_1,k_2,x}(y)=e^{\|x\|^2/2 -\|y\|^2 }.$$
We next check with dominated convergence  that
 for any  $h\in C_b(\mathbb{R}^n)$,
\begin{align}\label{density-b-limit-3a}
\lim_{k_1\to\infty}&\int_{ C_N^B-\sqrt{2k_1}\cdot \mathbf{1}}
 h(y)\cdot\prod_{i<j}(y_i-y_j)^{2k_2}\cdot\tilde f_{k_1,k_2,x}(y)\> dy \\
&=\int_{C_N^A} h(y)\cdot\prod_{i<j}(y_i-y_j)^{2k_2}\cdot e^{\|x\|^2/2 -\|y\|^2 } \> dy. 
\notag\end{align}
In order to check the assumptions of dominated convergence,
 we first again notice that for all $y\in C^B_N$  by Lagrange remainder there exists
 $w\in [0,1]$ such that 
\begin{equation}\label{Lagrange}
\ln \left(1+\frac{y_i}{\sqrt{2k_1}}\right)= \frac{y_i}{\sqrt{2k_1}}-\frac{y_i^2}{4k_1}\cdot w
\end{equation}
which ensures that
$$exp\left(-\|y\|^2/2-\sqrt{2 k_1}\cdot \sum_{i=1}^{N}y_i 
+2k_1\sum _{i=1}^N 
\ln(1+\frac{y_i}{\sqrt{2 k_1}})\right)$$
can be bounded above by  $ e^{-\|y\|^2/2}$.
Moreover, as in Section 3, we know from \cite{RV2} that the  Bessel functions $J_{(k_1,k_2)}^B$ 
admit the estimate
\begin{equation}
J_{(k_1,k_2)}^B(x,y+\sqrt{2k_1}\cdot \mathbf{1})\le 
\exp(\langle x,y+\sqrt{2k_1}\cdot \mathbf{1}\rangle). 
\end{equation}
This shows that
\begin{equation}\label{est-j-b2}
J_{(k_1,k_2)}^B(x,y+\sqrt{2k_1}\cdot\mathbf{1} )\le e^{2\langle x,y\rangle} 
\quad\quad\text{for}\>\> K_1>0,\> y\in C_N^B
\>\> \text{with }\>\>  \langle x,y\rangle\ge \langle x,\sqrt{2k_1}\cdot \mathbf{1} \rangle.
\end{equation}
On the other hand, as the density $\tilde f_{k_1,k_2,x}(y)$ is positive for $y$ only with 
 $ y/\sqrt{2k_1}+ \mathbf{1}    \in C_N^B$, and as
the set 
$$\{w\in  C_N^B:\> \langle x,w\rangle\le \langle x,\mathbf{1} \rangle\}$$
is compact, we obtain from Lemma 
\ref{B-A convergence}
that
$$\sup_{y\in\mathbb  R^N:\>  y/\sqrt{2k_1}+ \mathbf{1}\in C_N^B, \langle x, y/\sqrt{2k_1}\rangle\le \langle x,\mathbf{1} \rangle}
J_{(k_1,k_2)}^B(x,\sqrt{2k_1}(y/\sqrt{2k_1}+\mathbf{1} ))$$
is bounded.
Furthermore, 
$$\prod_{i<j}\left(\frac{y_i+y_j+2\sqrt{2 k_1}}{2\sqrt{2 k_1}}\right)^{2k_2}$$
can be estimated from above by some polynomial in $y$
 independent of  $k_1\ge1$,
 and the additional factor $\prod_{i<j}(y_i-y_j)^{2k_2}$ is also polynomially growing 
in $y$. Taking all upper estimates into account, we find an upper bound of the form
$  e^{-\|y\|^2/2} P(y)$ with some polynomial $P$. This ensures by dominated convergence that 
(\ref{density-b-limit-3a}) is in fact correct for all $h\in C_b(\mathbb{R}^n)$.

As the right hand side density
$$\prod_{i<j}(y_i-y_j)^{2k_2}\cdot e^{\|x\|^2/2 -\|y\|^2 }\quad\quad (y\in C_N^A)$$
in (\ref{density-b-limit-3a}) is the density of the distribution of
 $X^A_{1/2,k_2}$ (with start in the origin) up
 to normalization constants depending on $k_2,x,N$, we conclude for $h=1$,
that also the constants converge as needed for the theorem, and that the theorem holds.
 \end{proof}

\begin{remark}
We expect that the CLT \ref{Theorem-b-2} can be derived for Bessel processes $(X_{t,(k_1,k_2)})_{t\ge0}$ 
of type B with start in the origin $x=0$ also from the tridiagonal matrix models of Dumitriu and Edelman \cite{DE1} similar to the proofs
of the CLTs in  \cite{DE2} which correspond to the CLTs in Sections 1 and 2 above for $x=0$.
\end{remark}

\begin{remark}
The CLT \ref{Theorem-b-2}  for the starting point 0 should be compared with the CLT 4.1 in \cite{AV} where also Bessel processes
 $(X_{t,(k_1,k_2)})_{t\ge0}$ of type B are studied with $k_2$ fixed, $k_1\to\infty$, where there the starting points have the form 
$\sqrt{k_1}\cdot x$ with some fixed point $x$ in the interior on $C_N^B$. In Theorem 4.1 of \cite{AV},
 the limit distribution is an $N$-dimensional normal distribution with a covariance matrix depending on $x$.
 This limit is quite different from the limit in Theorem \ref{Theorem-b-2} above for $x=0$. 
This means that that the assertion of  Theorem 4.1 of \cite{AV} cannot be extended ``continuously'' 
from the interior of  $C_N^B$ to the origin.
 This is also clear by simple geometric considerations about the support of the limit measure.
It might be interesting to study CLTs similar to  Theorem 4.1 of \cite{AV} with starting points of the form $\sqrt{k_1}\cdot x$
with $x$ on the boundary of  $C_N^B$, but $x\ne0$.
\end{remark}

For particular values of $k_2$, namely $k=1/2,1,2$, Theorem  \ref{Theorem-b-2} above has some matrix-theoretic, or 
``geometric'', background:

\begin{remark}\label{remark-B2-long}
Fix one of the (skew-)fields $\mathbb F=\mathbb R, \mathbb C, \mathbb H$  with the real dimension $d=1,2,4$ respectively. For
integers $p\in\mathbb N$ consider the vector spaces $M_{p,N}(\mathbb F)$ of all $p\times N$-matrices over $\mathbb F$ with
 the real dimension $dpN$.
Choose the standard bases on these vector spaces such that we have  $d$ basis vectors in each entry.
Consider the $dpN$-dimensional associated Brownian motion
 $(B_t^p)_{t\ge0}$ on  $M_{p,N}(\mathbb F)$ starting in the origin. If we write $A^*:=\overline A^T\in M_{N,p}(\mathbb F)$ for matrices
 $A\in M_{p,N}(\mathbb F)$
 with the usual conjugation on $\mathbb F$, then the process $(Z_t^p:= (B_t^p)^*B_t^p)_{t\ge0}$ becomes 
a Wishart process on the closed cone $\Pi_N(\mathbb F)$
of all $N\times N$ positive semidefinite matrices over  $\mathbb F$  with shape parameter $p$; see \cite{Bru,DDMY} and references there
for details on Wishart processes.

Consider the spectral mapping $\sigma_N:\Pi_N(\mathbb F)\to C_N^B$ which assigns to each  matrix in $\Pi_N(\mathbb F)$
 its  ordered spectrum.
It is well-known that then $(\sqrt{\sigma_N(Z_t^p)})_{t\ge0}$ is a Bessel process on  $C_N^B$ of type $B_N$ with multiplicities
$$(k_1,k_2):=((p-N+1)\cdot d/2, d/2)$$
 where the symbol $\sqrt{.}$ means taking square roots in each component;
 see e.g. \cite{BF,R3} for details.

We thus conclude that Theorem  \ref{Theorem-b-2} for $k_2=1/2, 1,2$ corresponds to a CLT for Wishart distributions on  $\Pi_N(\mathbb F)$
with fixed time parameters where the shape parameters $p$ tend to $\infty$. To explain this CLT on the level of matrices,
we recapitulate that the distributions
 $\mu_t^p:=P_{Z_t^p}\in M^1(\Pi_N(\mathbb F))$ of $Z_t^p$ satisfy $\mu_t^{p_1}*\mu_t^{p_2}=\mu_t^{p_1+p_2}$ for $p_1,p_2\in\mathbb N$ with
 the usual convolution of measures on the vector space $\mathbb H_N(\mathbb F)$ 
of all $N\times N$ Hermitian matrices over $\mathbb F$ by the very construction of the 
random variables $Z_t^p$. Moreover, this convolution relation even remains valid  for all $p\in]0,\infty[$ which are sufficiently large. 
We thus may apply the classical LLNs and CLT for sums of iid random variables on finitely dimensional vector spaces to obtain
 LLNs and a CLT for Wishart distributions for $p\to\infty$.
A short computation in this setting for the CLT shows that here the centering (on the vector space
 $\mathbb H_N(\mathbb F)$) is performed with a multiple of the identity matrix, and that the associated centered limit normal distributions  
of dimension $\dim_{\mathbb R} \mathbb H_N(\mathbb F)$ have a simple shape. Using these data it can be shown that the image measure
of these  centered limit normal distributions on $\mathbb H_N(\mathbb F)$ under the spectral map 
 $\sigma_N:\mathbb H_N(\mathbb F)\to C_N^A$ is just a distribution of type A on $C_N^A$ as in Theorem  \ref{Theorem-b-2} above.
 We here omit further details.

Nevertheless, this reamrk shows that for $k_2=1/2, 1,2$,   Theorem  \ref{Theorem-b-2} corresponds to a CLT for Wishart distributions 
on $\mathbb H_N(\mathbb F)$ 
where the shape parameters tend to $\infty$.
\end{remark}

We finally remark that in the setting of Remark \ref{remark-B2-long},  
   Theorem  \ref{Theorem-b-2} for $k_2=1/2, 1,2$ is also related to limit theorems
 for radial random walks $(X_n^p)_{n\ge0}$ 
on  the vector spaces $M_{p,N}(\mathbb F)$ and their projections to $\Pi_N(\mathbb F)$ and $C_N^B$,
when the dimension parameter $p$ as well as the time parameter $n$ tend to $\infty$ in a 
coupled way; see \cite{G,RV4,V}.

As in the end of Section 3, Theorem \ref{Theorem-b-2} can be eaily extended from arbitrary, but fixed starting points
to arbitrary starting distributions on  $C_N^B$. 
 As the proof is the same as for Corollary \ref {corr-general-starting-b}, we omit the proof.

\begin{corollary}\label{corr-general-starting-b2}
Let $\mu\in  M^1( C_N^B)$ be an arbitrary starting distribution. Consider the Bessel processes
$(X_{t,(k_1,k_2)})_{t\ge0}$ of type B on  $C_N^B$ with starting distribution $\mu$.
Then, for all  $k_2>0$ and $t>0$, 
 $$ X_{t,(k_1,k_2)}-\sqrt{2t \cdot k_1}\cdot\mathbf{1}$$ 
converges  for $k_1\to \infty$ in distribution to
 $X^A_{t/2,k_2}$, where $(X^A_{s,k_2})_{s\ge 0}$ is a Bessel process of type A starting in the origin.
\end{corollary}

\section{A central limit theorem for the root system $D_{N}$}

We next briefly study limit theorems for Bessel processes of type $D_N$. We recapitulate that the root system
is given here by 
$$D_N=\{\pm e_1\pm e_j: \quad 1\le i<j\le N\}$$
with associated closed Weyl chamber 
$$C_N:=C_N^D=\{x\in\mathbb R^N: \quad x_1\ge \ldots\ge x_{N-1}\ge |x_N|\}.$$
$C_N^D$ may be seen as a doubling of $C_N^B$ w.r.t.~the last coordinate. We have a one-dimensional multiplicity $k\ge0$.
The generator of the transition semigroup of the Bessel process  $(X_{t,k})_{t\ge0}$ of type D is
\begin{equation}\label{def-L-D} Lf:= \frac{1}{2} \Delta f +
 k \sum_{i=1}^N \sum_{j\ne i} \Bigl( \frac{1}{x_i-x_j}+\frac{1}{x_i+x_j}  \Bigr)
 \frac{\partial}{\partial x_i}f, \end{equation}
where we again assume reflecting boundaries. Similar to the preceding cases, we have the transition probabilities
\begin{equation}\label{density-general-d}
K_{t,k}(x,A)=c_k^D \int_A \frac{1}{t^{\gamma_D+N/2}} e^{-(\|x\|^2+\|y\|^2)/(2t)} J_k^D(\frac{x}{\sqrt{t}}, \frac{y}{\sqrt{t}}) 
\cdot w_k^D(y)\> dy
\end{equation}
with
$$w_k^D(x):= \prod_{i<j}(x_i^2-x_j^2)^{2k}, \quad  \gamma_D:= kN(N-1);  $$
see \cite{Dem} for further details on this root system.
Furthermore, using the normalization (\ref{const-b}) for $k_2=k, k_1=0$, we see that the  normalization constant $c_k^D$ is given by
\begin{align}\label{const-d}
 c_k^D:=& \Bigl(\int_{C_N^D}  e^{-\|y\|^2/2} w_k^D(y) \> dy\Bigr)^{-1} \\
=&\frac{N!}{2^{N(N-1)k-N/2+1}} \cdot\prod_{j=1}^{N}\frac{\Gamma(1+k)}{\Gamma(1+jk)\Gamma(\frac{1}{2}+(j-1)k)}. 
\notag\end{align}

We now proceed similar to Section 4 of \cite{AV}.
Using the known explicit representation
$$L_N^{(\alpha)}(x):=\sum_{k=0}^N { N+\alpha\choose N-k}\frac{(-x)^k}{k!}$$
of the Laguerre polynomials according to (5.1.6) of \cite{S}, we can form the polynomial  $L_N^{(-1)}$ of order $N\ge1$
where,
 by (5.2.1) of \cite{S}, 
\begin{equation}\label{laguerre-1}
L_N^{(-1)}(x)=-\frac{x}{N}L_{N-1}^{(1)}(x).
\end{equation}
Continuity arguments thus show that the equivalence of (2) and (3) of
 Lemma \ref{char-zero-B1} remains valid also for $\nu=0$  by using the
$N$ different zeros $z_1>\ldots>z_N=0$ of $L_N^{(-1)}$.
With these notations we have the following fact which is similar to the results above in the cases A and B;
 see Section 4 of \cite{AV}:

\begin{lemma}\label{char-zero-D}
 For $r\in C_N^D$, the following statements  are equivalent:
\begin{enumerate}
\item[\rm{(1)}] The function 
$W_D(y):=2\sum_{ i<j} \ln(y_i^2-y_j^2) -\|y\|^2/2$
 is maximal in $ r\in C_N^B$;
\item[\rm{(2)}] $r_N=0$, and
for $i=1,\ldots,N-1$, 
$$4 \sum_{j: j\ne i} \frac{1}{r_i^2-r_j^2} =1;$$ 
\item[\rm{(3)}] If $z_1^{(1)}>\ldots>z_{N-1}^{(1)}>0$ are the $N-1$ ordered zeros of 
 the classical  Laguerre polynomial $L_{N-1}^{(1)}$, then 
\begin{equation}\label{y-max-D}
2\cdot (z_1^{(1)},\ldots, z_{N-1}^{(1)},0)= (r_1^2,\ldots,r_N^2).
\end{equation}
\end{enumerate}
\end{lemma}

We now turn to the  main result of this section:

\begin{theorem}\label{clt-main-D}
Consider the  Bessel processes $(X_{t,k})_{t\ge0}$ of type $D_N$ on $C_N^D$ with multiplicity $k>0$ which start in $0$.
Then, for the vector $r\in C_N^D$ of Lemma \ref{char-zero-D},
$$\frac{X_{t,k}}{\sqrt t} -  \sqrt{k }\cdot r$$
converges for $k\to\infty$ to the centered $N$-dimensional distribution $N(0,t\cdot \Sigma)$
with the regular covariance matrix $\Sigma$ with  $\Sigma^{-1}=(s_{i,j})_{i,j=1,\ldots,N}$ with
\begin{equation}\label{covariance-matrix-D}
s_{i,j}:=\left\{ \begin{array}{r@{\quad\quad}l}  1+ 2\sum_{l\ne i} (r_i-r_l)^{-2}+2\sum_{l\ne i} (r_i+r_l)^{-2}   &
 \text{for}\quad i=j \\
 2(r_i+r_j)^{-2}  -2(r_i-r_j)^{-2} & \text{for}\quad i\ne j  \end{array}  \right.  . 
\end{equation}
\end{theorem}

\begin{proof} As in the preceding cases we  assume that $t=1$ without loss of generality.
$X_{1,k}$ has the  density
$$c_k^D e^{-\|y\|^2/2}
\cdot exp\Bigl( 2k \sum_{i<j}\ln(y_i^2-y_j^2) \Bigr)$$
on $C_N^D$. Hence, $X_{1,k}-\sqrt{k}\cdot r$ has the density
\begin{align}\label{d-density-detail}
f_k^D(y):=&c_k^D e^{-\|y+\sqrt{k}\cdot r \|^2/2}
 exp\Bigl(2k \sum_{i<j}\ln\bigl((y_i+\sqrt{k}\cdot r_i)^2- (y_j+\sqrt{k}\cdot r_j)^2\bigr)\Bigr)\notag\\
=c_k^D &
  exp\Bigl(2k \sum_{i<j}\ln\bigl(1+ \frac{y_i-y_j}{\sqrt{k}(r_i-r_j)}\bigr)+
2k \sum_{i<j}\ln\bigl(1+ \frac{y_i+y_j}{\sqrt{k}(r_i+r_j)}\bigr)\Bigr) \times \notag\\
&  e^{-\|y\|^2/2} e^{-k\|r\|^2/2}  e^{-\sqrt{k} \langle y,r\rangle} exp\Bigl(2k \sum_{i<j}\Bigl(\ln(\sqrt{k}( r_i-r_j)) +\ln(\sqrt{k}( r_i+r_j))\Bigr)\Bigr)
 \notag
\end{align}
on the shifted cone $C_N^D-\sqrt{k}\cdot r$ with $f_k^D(y)=0$ otherwise on $\mathbb R^N$.
We again split $f_k^D(y)$ into 
$$f_k^D(y)= \tilde c_k \cdot h_k(y),$$
 where $ h_k(y)$ depends on $y$, and 
 $\tilde c_k$   is  constant w.r.t.~$y$.
The part depending on $y$ is given by
\begin{equation}h_k(y):=  exp\Bigl(-\|y\|^2/2 -\sqrt{k} \langle y,r\rangle +
2k \sum_{i<j}\Bigl(\ln\bigl(1+ \frac{y_i-y_j}{\sqrt{k}(r_i-r_j)}\bigr)+\ln\bigl(1+ \frac{y_i+y_j}{\sqrt{k}(r_i+r_j)}\bigr)\Bigr)\Bigr).
\notag\end{equation}
The expansion of $\ln(1+x)$ together with Lemma \ref{char-zero-D}      yields as in the proof of Theorem \ref{clt-main-b1} that
\begin{equation}\label{summary-d}
\lim_{k\to\infty}h_k(y)= exp\Bigl(-\frac{\|y\|_2^2}{2}-\sum_{i<j}\frac{(y_i-y_j)^2}{(r_i-r_j)^2} 
-\sum_{i<j}\frac{(y_i+y_j)^2}{(r_i+r_j)^2}\Bigr).
\end{equation}
Moreover, as in the proofs of  Theorems \ref{clt-main-b1} and \ref{clt-main-a}, we see that for all  $f\in C_b(\mathbb R^N)$,
\begin{align}\label{density-d-limit-2}
\lim_{k\to\infty}&\int_{\mathbb R^N} f(y)\cdot h_k(y)\> dy = 
 \\
=
& \int_{\mathbb R^N} f(y) \> exp\Bigl(-\Bigl(\frac{\|y\|^2}{2}+\sum_{i<j}\frac{(y_i-y_j)^2}{(r_i-r_j)^2}
+\sum_{i<j}\frac{(y_i+y_j)^2}{(r_i+r_j)^2}
\Bigr)\Bigr) \> dy.\notag
\end{align}
This implies the theorem as in the proof of  Theorem \ref{clt-main-b1}.
\end{proof}

Let $(X_{t,k}^D)_{t\ge0}$ be a Bessel process of type D with multiplicity $k\ge0$ on the chamber $C_N^D$. Then the process
$(X_{t,k}^B)_{t\ge0}$ with 
$$X_{t,k}^{B,i}:= X_{t,k}^{D,i} \quad(i=1,\ldots,N-1), \quad X_{t,k}^{B,N}:= |X_{t,k}^{D,N} |$$
is a Bessel process of type B with the multiplicity $(k_1,k_2):=(0,k)$. This follows easily from a comparison of the
corresponding generators and was also observed in \cite{AV}. 
The central limit theorem \ref{clt-main-D} for $(X_{t,k}^D)_{t\ge0}$ thus leads to the following ``one-sided CLT'' for
 Bessel processes of type B with the multiplicities $(0,k)$ for $k\to\infty$:

\begin{corollary}\label{clt-main-b-one-sided}
Consider the  Bessel processes $(X_{t,k})_{t\ge0}$ of type $B_N$ on $C_N^B$ with multiplicities $(0,k)$  which start in $0$.
Then, for the vector $r$ of Lemma \ref{char-zero-D} on the boundary of  $C_N^B$,
$$\frac{X_{t,k}}{\sqrt t} -  \sqrt{k }\cdot r$$
converges for $k\to\infty$ in distribution to the ``one-sided normal distribution'' on the half space 
$$H_N:=\{x\in\mathbb R^N: x_N\ge0\}$$
which appears as image of the
centered $N$-dimensional distribution $N(0,t\cdot \Sigma)$
with the regular covariance matrix $\Sigma$ as in Theorem  \ref{clt-main-D} under the mapping
$$\mathbb R^N  \longrightarrow H_N, \quad (x_1,\ldots,x_N)\mapsto (x_1,\ldots,x_{N-1}, |x_N|).$$
\end{corollary}

\section{A third central limit theorem for the root system $B_N$}

In this section we again study Bessel processes of type B with multiplicities $k=(k_1,k_2)$.
 We here fix $k_1>0$ and consider $k_2\to\infty$. As far as we know,
 this limit was not considered in the literature up to now. It will turn out that this case is closely
 related to the limits in the case D above and in particular to 
the B-case for the multiplicities $(0,k_2)$ for $k_2\to\infty$ in Corollary \ref{clt-main-b-one-sided}.

We start with a law of large numbers which corresponds to limit results of \cite{AKM1}, \cite{AKM2}
  for the root systems of type A and B. For this we fix  $k_1>0$ and consider the densities $f_{t,k}$ of 
$X_{t,k}/\sqrt{tk_2}$ of the Bessel processes $(X_{t,k})_{t\ge0}$ of type $B_N$ on $C_N^B$
 where we suppose that the processes start in the origin. Then, by the scaling properties of the $X_{t,k}$ and by
(\ref{density-general-b}), these densities are independent from $t>0$ and have the form
$$f_{1,k}(x)= const(k)\cdot\prod_{i=1}^N x_i^{2k_1}\cdot exp\bigl( k_2 W(x)\bigr)$$
with some normalization constant $const(k)>0$ (see Section 2 for the details), and with
$$W(x)= - \|x\|_2^2/2 +2\sum_{i<j}\ln(x_i^2-x_j^2) \quad\quad\text{for}\quad\text x\in C_N^B.$$
We know from Lemma \ref{char-zero-D} that on $C_N^B$, the function $W$ admits an
 unique maximum which is located at
$r=(r_1,\ldots,r_n)\in C_N^B$ with
$$(r_1^2,\ldots,r_N^2)= 2\cdot (z_1^{(1)},\ldots, z_{N-1}^{(1)},0)$$
for the $N-1$ ordered zeros
 $z_1^{(1)}>\ldots>z_{N-1}^{(1)}>0$  of  $L_{N-1}^{(1)}$.
This optimal point $r$ is in the support of the measure ${\bf 1}_{C_N^B} \cdot\prod_{i=1}^N x_i^{2k_1}\> dx$.
Some standard arguments from analysis now readily lead to the following limit law.

\begin{proposition}\label{clt-main-b-one-sided-general--lln} For each 
 $k_1>0$ and $t>0$, $X_{t,(k_1,k_2)}/\sqrt{ tk_2}$ converges for $k_2\to\infty$ 
in distribution to 
the point measure $\delta_r$ for $r\in C_N^B$ as above.
\end{proposition}

We here skip details of the proof, as this proposition follows immediately 
from the following associated CLT which is analog to the CLTs in the previous sections and in particular to 
Corollary \ref{clt-main-b-one-sided}. Notice that the proof of Theorem \ref{clt-main-b-one-sided-general} does not rely on
Proposition \ref{clt-main-b-one-sided-general--lln}

\begin{theorem}\label{clt-main-b-one-sided-general}
Consider the  Bessel processes $(X_{t,k})_{t\ge0}$ of type $B_N$ on $C_N^B$ with multiplicities $(k_1,k_2)$  which start in $0$.
Fix  $k_1\ge 0$ and $t>0$. Then, for $k_2\to\infty$,
$$\frac{X_{t,(k_1,k_2)}}{\sqrt t} -  \sqrt{k_2 }\cdot r$$
converges  in distribution to the ``one-sided normal distribution'' on the half space 
$H_N:=\{x\in\mathbb R^N: x_N\ge0\}$ which appears in Corollary \ref{clt-main-b-one-sided}.
\end{theorem}

\begin{proof} As in the preceding cases we  assume $t=1$.
The case  $k_1= 0$ is shown in Corollary \ref{clt-main-b-one-sided}.
 
The case $k_1> 0$ can be proved similar as in Theorem \ref{clt-main-D}. In fact,
the shifted random variables
 $X_{1,(k_1,k_2)}-\sqrt{k_2}\cdot r$ have  densities of the form
\begin{align}\label{d-density-detail}
f_{(k_1,k_2)}(y):=&\tilde c(k_1,k_2)\cdot \prod_{i=1}^N (y_i+\sqrt{k_2}\cdot r_i)^{2k_1}\cdot
 e^{-\|y+\sqrt{k_2}\cdot r \|^2/2}\notag\\
&\quad\quad\times
 exp\Bigl(2k_2 \sum_{i<j}\ln\bigl((y_i+\sqrt{k_2}\cdot r_i)^2- (y_j+\sqrt{k_2}\cdot r_j)^2\bigr)\Bigr)
\notag\\
=&c(k_1,k_2)\cdot \prod_{i=1}^N \Biggl(\frac{y_i+\sqrt{k_2}\cdot r_i}{\sqrt{k_2}\cdot r_i}\Biggr)^{2k_1}
\cdot h_{k_2}(y)
\end{align}
on the shifted cone $C_N^B-\sqrt{k_2}\cdot r$ 
with 
\begin{align}h_{k_2}(y):=  exp\Bigl(&-\|y\|^2/2 -\sqrt{k_2} \langle y,r\rangle + \notag\\
+&2k_2 \sum_{i<j}\Bigl(\ln\bigl(1+ \frac{y_i-y_j}{\sqrt{k_2}(r_i-r_j)}\bigr)+
\ln\bigl(1+ \frac{y_i+y_j}{\sqrt{k_2}(r_i+r_j)}\bigr)\Bigr)\Bigr)
\notag\end{align}
where $f_{(k_1,k_2)}(y)=0$ otherwise on $\mathbb R^N$ and where $\tilde c(k_1,k_2), c(k_1,k_2)>0$ are suitable 
constants.
 As in Theorem \ref{clt-main-D} we obtain that
\begin{equation}\label{summary-b3}
\lim_{k_2\to\infty}h_{k_2}(y)= exp\Bigl(-\frac{\|y\|_2^2}{2}-\sum_{i<j}\frac{(y_i-y_j)^2}{(r_i-r_j)^2} 
-\sum_{i<j}\frac{(y_i+y_j)^2}{(r_i+r_j)^2}\Bigr).
\end{equation}

On the other hand, 
$$\Biggl( \prod_{i=1}^N\frac{y_i+\sqrt{k_2}\cdot r_i}{\sqrt{k_2}\cdot r_i}\Biggr)^{2k_1}$$
tends to 1 for $k_2\to\infty$ and is bounded by an expression of the form
$$c(k_1) +d(k_1)(y_1y_2\ldots y_N)^{2k_1} $$
for all $y$ independent of  $k_2\ge1$ with suitable constants $c(k_1),d(k_1)>0$. 
 These (in $y$) polynomial bounds do not affect the fact that we still may apply 
dominated convergence theorem like in the proofs of 
  Theorems \ref{clt-main-b1}, \ref{clt-main-a}, and \ref{clt-main-D} due to the Gaussian bounds for
$h_{k_2}(y)$ in the proofs there. Moreover, as the shifted cones $C_N^B-\sqrt{k_2}\cdot r$ converge to
$\mathbb R^{N-1}\times [0,\infty[$ for $k_2\to\infty$ due to $r_N=0$,
we thus conclude that
 for all  $f\in C_b(\mathbb R^N)$,
\begin{align}\label{density-b-limit-3}
\lim_{k_2\to\infty}&\int_{C_N^B-\sqrt{k_2}\cdot r} f(y)\cdot 
\Biggl( \prod_{i=1}^N\frac{y_i+\sqrt{k_2}\cdot r_i}{\sqrt{k_2}\cdot r_i}\Biggr)^{2k_1}\cdot h_{k_2}(y)\> dy = 
 \\
=
& \int_{\mathbb R^{N-1}\times [0,\infty[} f(y) 
\> exp\Bigl(-\Bigl(\frac{\|y\|^2}{2}+\sum_{i<j}\frac{(y_i-y_j)^2}{(r_i-r_j)^2}
+\sum_{i<j}\frac{(y_i+y_j)^2}{(r_i+r_j)^2}
\Bigr)\Bigr) \> dy.\notag
\end{align}
If we look into this formula for $f\equiv 1$, we obtain that the corresponding normalization constants 
converge as needed for our CLT, and the theorem follows.
\end{proof}

\begin{remark}
It can be easily seen that the limit relation for Bessel functions of type B in Lemma \ref{asymptotic-bessel-B1}
is also available for $\nu=0$ and $N\ge2$. This shows that Theorem \ref{clt-main-b-one-sided-general}
can be extended from the starting point 0 to an arbitrary starting point $x\in C_N^B$.
 The details of the proof are the same as in the proofs of Theorems \ref{clt-main-b1}   and \ref{Theorem-b-2}.
\end{remark}


\begin{thebibliography}{999}

\bibitem[A]{A} B. Amri, Note on Bessel functions of type $A_{N-1}$. 
\textit{Integral Transforms and Special Functions} 25 (2014),  448-461.

\bibitem[AM]{AM} S. Andraus, S. Miyashita: Two-step  asymptotics of scaled Dunkl processes.
\textit{J.  Math. Phys.} 56  (2015) 103302.

\bibitem[AKM1]{AKM1} S. Andraus, M. Katori, S. Miyashita, Interacting particles on the line 
and Dunkl intertwining operator of type $A$: Application to the freezing regime. 
\textit{J. Phys. A: Math. Theor. } 45  (2012) 395201.

\bibitem[AKM2]{AKM2} S. Andraus, M. Katori, S. Miyashita, Two limiting regimes of interacting Bessel processes. 
 \textit{J. Phys. A: Math. Theor. } 47  (2014) 235201.


\bibitem[AV]{AV} S. Andraus, M. Voit, Limit theorems
 for multivariate Bessel processes in the freezing regime, Preprint 2018. arXiv:1804.03856 

\bibitem[An]{An} J.-P. Anker.  An introduction to Dunkl theory and its analytic aspects. In: G. Filipuk,
Y. Haraoka, S. Michalik. Analytic, Algebraic and Geometric Aspects of Differential Equations,
Birkh{\"a}user, pp.3-58, 2017.
 
 \bibitem[BF]{BF} T.H. Baker, P.J. Forrester, The Calogero-Sutherland model and generalized classical
polynomials. \textit{Comm. Math. Phys.} 188 (1997), 175--216.


\bibitem[Bru]{Bru} M. Bru, Wishart processes. \textit{J. Theoret. Probab.} 4 (1991), 725--751.


\bibitem[CGY]{CGY} O. Chybiryakov, L. Gallardo, M. Yor, Dunkl processes and
 their radial parts relative to a root system. In:
 P. Graczyk et al. (eds.), Harmonic and Stochastic analysis of Dunkl processes. Hermann, Paris 2008.


\bibitem[D]{D} P. Deift, Orthogonal Polynomials and Random Matrices:
 A Riemann-Hilbert Approach. Amer. Math. Soc. 2000.


\bibitem[Dem]{Dem} N. Demni, Generalized Bessel function of type D. 	
SIGMA 4 (2008), 075, 7 pages, arXiv:0811.0507. 

\bibitem[DF]{DF} P. Desrosiers, P. Forrester,
 Hermite and Laguerre $\beta$-ensembles: Asymptotic corrections to the eigenvalue density. \textit{Nuclear Physics B}
743 (2006), 307-332.


\bibitem[DDMY]{DDMY} C. Donati-Martin, Y. Doumerc, H. Matsumoto, M. Yor, 
Some properties of the Wishart processes and a matrix extension of the Hartman-Watson law. 
\textit{Publ. Math. RIMS Kyoto} 40 (2004), 1385-1412.



\bibitem[DV]{DV} J.F. van Diejen, L. Vinet, Calogero-Sutherland-Moser Models.
 CRM Series in Mathematical Physics, Springer-Verlag 2000.

\bibitem[DE1]{DE1} I. Dumitriu, A. Edelman, Matrix models for beta-ensembles. \textit{ J. Math. Phys.} 43 (2002),  5830-5847.

\bibitem[DE2]{DE2} I. Dumitriu, A. Edelman, Eigenvalues of Hermite and Laguerre ensembles: large beta asymptotics,
\textit{Ann. Inst. Henri Poincare (B)} 41 (2005), 1083-1099.


\bibitem[GY]{GY} L. Gallardo, M. Yor, Some remarkable properties of the Dunkl martingale. In:
 Seminaire de Probabilites XXXiX, pp. 337-356, dedicated to P.A. Meyer, vol. 1874,
 Lecture Notes in Mathematics, Springer 2006.

\bibitem[G]{G} W. Grundmann, Limit theorems for radial random walks on Euclidean spaces
 of high dimensions. 
\textit{J. Austral. Math. Soc.} 97 (2014),  212--236.

\bibitem[Mac]{Mac} I. Macdonald, Some conjectures for root systems,\textit{ SIAM Journal Math. Anal.} 13 (1982), 988--1007.

\bibitem[Me]{Me} M. Mehta, Random matrices (3rd ed.), Elsevier/Academic Press, Amsterdam, 2004. 

\bibitem[OO]{OO}  A. Okounkov, G. Olshanski,
Shifted Jack polynomials, binomial formula, and applications, 
\textit{Math. Res. Letters} 4 (1997), 69–-78.

\bibitem[O]{O} E. Opdam, Some applications of hypergeometric shift operators,\textit{ Invent. Math.} 98 (1989), 275--282.

\bibitem[RY]{RY} D. Revuz, M. Yor, Continuous Martingales and Brownian Motion. Springer 1999.

\bibitem[R1]{R1} M. R\"osler,
Generalized Hermite polynomials and the heat equation for Dunkl operators.
\textit{Comm. Math. Phys.} 192 (1998),  519-542.

\bibitem[R2]{R2} M. R\"osler, Dunkl operators: Theory and applications.
In: Orthogonal polynomials and special functions, Leuven 2002, \textit{Lecture Notes in Math.} 1817 (2003), 93--135.

\bibitem[R3]{R3} M. R\"osler, Bessel convolutions on matrix cones. \textit{Compos. Math.} 143 (2007), 749-779.

\bibitem[RV1]{RV1} M. R\"osler, M. Voit, Markov processes related with Dunkl operators.
\textit{Adv. Appl. Math.}  21 (1998) 575--643.

\bibitem[RV2]{RV2} M. R\"osler, M. Voit, Dunkl theory, convolution algebras, and related Markov processes.
 In: P. Graczyk et al. (eds.), Harmonic and Stochastic analysis of Dunkl processes. Hermann, Paris 2008.


\bibitem[RV3]{RV3} M. R\"osler, M. Voit, A limit relation for Dunkl-Bessel functions
of type A and B. \textit{ SIGMA} 4 (2008), 083, 9 pages.

\bibitem[RV4]{RV4} M. R\"osler, M. Voit,  Limit theorems for radial random walks on $p \times q$--matrices as $p$ tends to infinity.
 \textit{Math. Nachr.} 284 (2011), 87-104

\bibitem[S]{S} G. Szeg{\"o}, Orthogonal Polynomials. 
Colloquium Publications (American Mathematical Society), Providence, 1939.

\bibitem[V]{V} M. Voit, 
Central limit theorems for radial random walks on $p\times q$ matrices for $p\to\infty$. 
\textit{Adv. Pure Appl. Math.} 3 (2012)  231-246.

 \bibitem[VW]{VW} M. Voit, J. Woerner,
 Central limit theorems for multivariate Bessel processes in the
 freezing regime for varying starting points. Preprint 2018.
\end{thebibliography}
\end{document}